\documentclass[12pt]{article}

\usepackage{amsmath,amssymb,amsbsy,amsfonts,amsthm,latexsym,              amsopn,amstext,amsxtra,euscript,amscd,mathrsfs,graphics,epsfig,color}

\begin{document}

\newtheorem{theorem}{Theorem}
\newtheorem{lemma}[theorem]{Lemma}
\newtheorem{corollaire}[theorem]{Corollaire}
\newtheorem{remarque}{Remarque}

\newtheorem{proposition}[theorem]{Proposition}
\newtheorem{corollary}[theorem]{Corollary}
\newtheorem{conjecture}{Conjecture}
\theoremstyle{definition}
\newtheorem*{definition}{Definition}
\newtheorem{remark}{Remark}
\newtheorem*{example}{Example}

\def\cA{\mathcal A}
\def\cB{\mathcal B}
\def\cC{\mathcal C}
\def\cD{\mathcal D}
\def\cE{\mathcal E}
\def\cF{\mathcal F}
\def\cG{\mathcal G}
\def\cH{\mathcal H}
\def\cI{\mathcal I}
\def\cJ{\mathcal J}
\def\cK{\mathcal K}
\def\cL{\mathcal L}
\def\cM{\mathcal M}
\def\cN{\mathcal N}
\def\cO{\mathcal O}
\def\cP{\mathcal P}
\def\cQ{\mathcal Q}
\def\cR{\mathcal R}
\def\cS{\mathcal S}
\def\cU{\mathcal U}
\def\cT{\mathcal T}
\def\cV{\mathcal V}
\def\cW{\mathcal W}
\def\cX{\mathcal X}

\def\Z{{\mathbb Z}}
\def\R{{\mathbb R}}
\def\F{{\mathbb F}}
\def\N{{\mathbb N}}
\def\C{{\mathbb C}}
\def\Q{{\mathbb Q}}

\def\Fp{\F_p}
\def\e{\mathbf{e}}
\def\ep{\mathbf{e}_p}
\def\tp{\texttt{p}}
\def\tq{\texttt{q}}
\def\eps{{\varepsilon}}
\def\mand{\qquad \mbox{and} \qquad}
\def\scr{\scriptstyle}
\def\\{\cr}
\def\({\left(}
\def\){\right)}
\def\[{\left[}
\def\]{\right]}
\def\<{\langle}
\def\>{\rangle}
\def\fl#1{\left\lfloor#1\right\rfloor}
\def\rf#1{\left\lceil#1\right\rceil}
\def\le{\leqslant}
\def\ge{\geqslant}

\def\Sp{\cS_\pi}
\def\Sq{\cS_\Pi}
\def\NP{{\mathscr N}_\cP}
\def\sP{{\mathscr P}}
\def\tNP{\widetilde {\mathscr N}_\cP}

\def\xxx{\vskip5pt\hrule\vskip5pt}
\def\imhere{ \xxx\centerline{\sc I'm here}\xxx }

\def\mand{\qquad \mbox{and} \qquad}

\newcommand{\comm}[1]{\marginpar{%
\vskip-\baselineskip \raggedright\footnotesize
\itshape\hrule\smallskip#1\par\smallskip\hrule}}

\title{\bf On group structures realized by 
elliptic curves over arbitrary finite fields}

\author{
{\sc William D.~Banks} \\
{Department of Mathematics, University of Missouri} \\
{Columbia, MO 65211 USA} \\
{\tt bankswd@missouri.edu} 
\and
{\sc Francesco Pappalardi} \\
{Dipartimento di Matematica, Universit\`a Roma Tre} \\
{Roma, I--00146, Italy} \\
{\tt pappa@mat.uniroma3.it}
\and
{\sc Igor E.~Shparlinski} \\
{Department of Computing, Macquarie University} \\
{Sydney, NSW 2109, Australia} \\
{\tt igor@ics.mq.edu.au}}

\date{\today}
\pagenumbering{arabic}

\maketitle

\begin{abstract}
We study the collection of group structures
that can be realized as a group of rational points on 
an elliptic curve over a finite field (such groups
are well known to be of rank at most two). We also study various
subsets of this collection which correspond to curves over prime
fields or to curves with a prescribed torsion. Some of our 
results are rigorous and are based on recent advances 
in analytic number theory, some are conditional under 
certain widely believed conjectures, and others are purely
heuristic in nature.
\end{abstract}

\newpage

\section{Introduction}
\label{sec:intro}

Let $\F_q$ denote the finite field with $q$ elements.
It is well known that the group $E(\F_q)$ of points on an elliptic
curve $E$ defined over $\F_q$ has rank at most two, and therefore,
\begin{equation}
\label{eq:EFq} E(\F_q)\cong\Z_n\times\Z_{kn}
\end{equation}
for some natural numbers $n$ and $k$, where $\Z_m$ denotes the ring
of congruence classes modulo $m$ for each natural number $m$; 
see~\cite{How,R,Vol,Wa}.  On the
other hand, little is known about the structure of the
set of groups $\Z_n\times\Z_{kn}$ that can be realized
as the group of points on an elliptic curve defined over a finite field.
For this reason, we introduce and investigate the set
$$
\Sq=\bigl\{(n,k)\in\N^2~:~\exists~\text{prime
power~}q\text{~and~}E/\F_q \text{~with~}
E(\F_q)\cong\Z_n\times\Z_{kn}\bigr\}.
$$
We are also interested in groups $\Z_n\times\Z_{kn}$ with a
realization~\eqref{eq:EFq} in which $q=p$ is a prime number, hence
we study the subset $\Sp\subset\Sq$ defined by
$$
\Sp=\bigl\{(n,k)\in\N^2~:~\exists~\text{prime~}p\text{~and~}E/\F_p
\text{~with~} E(\F_p)\cong\Z_n\times\Z_{kn}\bigr\}.
$$

Although one can expect $\Sp$ and $\Sq$ to be reasonably ``dense'' in
$\N^2$, the complementary sets also appear to be rather large. For
example, here is the list of pairs $(n,k)\not\in\Sq$ with $n,k\le 25$:
\begin{equation}
\begin{split}
\label{eq:misses}
(11,1),(11,14),(13,6),(13,25),(15,4),\hskip 18pt\\
(19,7),(19,10),(19,14),(19,15),(19,18),\hskip 10pt\\
(21,18),(23,1),(23,5),(23,8),(23,19),(25,5),(25,14).\hskip -25pt
\end{split}
\end{equation}

To investigate the distribution in $\N^2$ of the elements of
$\Sp$ and of $\Sq$, for natural numbers $N$ and $K$ we introduce the sets
\begin{equation*}
\begin{split}
\Sp(N,K)&=\bigl\{(n,k)\in\Sp~:~n\le N,~k\le K\bigr\},\\
\Sq(N,K)&=\bigl\{(n,k)\in\Sq~:~n\le N,~k\le K\bigr\}.
\end{split}
\end{equation*}
These sets are the main objects of study in this note.

For natural numbers $n$ and $k$, we also put
$$
\cP(n,k)=\bigl\{\text{primes~}p~:~\exists~E/\F_p\text{~for
which~}E(\F_p)\cong \Z_n\times\Z_{kn}\bigr\}.
$$
The set $\cP(n,k)$ parametrizes the set of finite fields of prime
cardinality over which $\Z_n\times\Z_{kn}$ can be
realized as the group of points on an elliptic curve.  For natural
numbers $N$ and $K$ we study the double sum
$$
\NP(N,K)=\sum_{n\le N}\sum_{k\le K}\#\cP(n,k),
$$
for which we obtain an asymptotic formula in certain ranges.

Finally, for natural numbers $m,k$ we introduce and compare the sets
\begin{equation*}
\begin{split}
\cN_{m,k}&=\bigl\{n\in\N~:~\exists~p\text{~prime and~}
E/\F_{p^m}\text{~with~}E(\F_{p^m})\cong\Z_n\times\Z_{kn}\bigr\},\\
\widetilde\cN_{m,k}&=\bigl\{n\in\N~:~\exists~p\text{~prime,~}
\ell\in\Z\text{~with~}p^m=kn^2+\ell n+1,~|\ell|\le 2\sqrt{k}\,\bigr\}.
\end{split}
\end{equation*}

We remark that the distribution of group structures generated by
elliptic curves over a \emph{fixed} finite field $\F_q$ has been 
studied in~\cite{RFS}.

\section{Notational conventions}

Throughout the paper, the letter $p$ always denotes a prime number,
and $q$ always denotes a prime power. As usual, we use $\pi(x)$ to denote 
the number of $p \le x$. For coprime integers $a$ and
$m\ge 1$, we put
\begin{equation*}
\begin{split}
\pi(x;m,a)&=\#\bigl\{p\le x~:~p \equiv a \pmod m\bigr\},\\
\Pi(x;m,a)&=\#\bigl\{q\le x~:~q \equiv a \pmod m\bigr\}.
\end{split}
\end{equation*}
We also set
$$
\psi(x;m,a)=\sum_{\substack{n\le x\\n\equiv a\hskip-7pt\pmod m}}\Lambda(n),
$$
where $\Lambda(n)$ is the von Mangoldt function.

For any set $\cA\subseteq\N$ and real $x>0$,
we denote $\cA(x)=\bigl\{a\in\cA~:~a\le x\bigr\}$. 

For functions $F$ and $G>0$ the notations $F=O(G)$, $F\ll G$, and $G\gg F$
are all equivalent to the assertion that the inequality $|F|\le c\,G$ holds
with some constant $c>0$.  In what follows, all constants implied by the
symbols $O$, $\ll$, and $\gg$ may depend (where obvious) on the small real
parameter $\eps$ but are absolute otherwise; we write $O_\rho$, $\ll_\rho$,
and $\gg_\rho$ to indicate that the implied constant depends on a given
parameter $\rho$. 

\section{Preliminaries}

\begin{lemma}
\label{lem:character} If $q$ is a prime power, and $E$ is an elliptic
curve defined over $\F_q$ such that $E(\F_q)\cong\Z_n\times\Z_{kn}$,
then $q=kn^2+\ell n+1$ for some integer $\ell$ that satisfies
$|\ell|\le 2\sqrt{k}$.
\end{lemma}

\begin{proof}
By the Hasse bound, we can write $kn^2=q+1-a$ for some integer
$a$ that satisfies the bound $a^2\le 4q$. Using the Weil pairing
one also sees that $q\equiv 1\pmod{n}$, hence $a=\ell n+2$ for some
integer $\ell$, and we have $q=kn^2+\ell n+1$. Since
$$
\ell^2n^2+4\ell n+4=(\ell n+2)^2=a^2\le 4q=4kn^2+4\ell n+4,
$$
it follows that $|\ell|\le 2\sqrt{k}$ as required.
\end{proof}

The following result of Waterhouse~\cite{Wa} (see
also~\cite[Theorems~4.3]{W}) is a characterization of the
natural numbers $N$ that can be realized as the cardinality of
the group of $\F_q$-rational points on an elliptic curve $E$
defined over $\F_q$. 

\begin{lemma} 
\label{lem:Water}
Let $q=p^m$ be a prime power, and suppose that $N=q+1-a$ for some integer $a$.
Then, there is an elliptic curve $E$ defined over $\F_q$ such that $\#E(\F_q)=N$
if and only if $|a|\le 2\sqrt q$ and one of the following conditions is met:
\begin{itemize} 
\item[$(i)$] $\gcd(a, p) = 1$;
\item[$(ii)$] $m$ even and $a=\pm 2\sqrt q$;
\item[$(iii)$] $m$ is even, $p\not\equiv 1 \pmod 3$, and $a = \pm\sqrt q$;
\item[$(iv)$] $m$ is odd, $p = 2$ or $3$, and $a = \pm p^{(m+1)/2}$;
\item[$(v)$] $m$ is even, $p\not\equiv 1 \pmod 4$, and $a = 0$;
\item[$(vi)$] $m$ is odd and $a = 0$.
\end{itemize}
\end{lemma}

For every admissible cardinality $N$, the following result of R\"uck~\cite{R}
(see also~\cite[Theorems~4.4]{W})
describes the group structures that are possible for $E(\F_q)$ given that
$\#E(\F_q)=N$; see also~\cite{How,Vol}.

\begin{lemma}
\label{lem:Ruck}
Let $q=p^m$ be a prime power, and suppose that $N$ is an integer
such that $\#E(\F_q)=N$ for some elliptic curve $E$ defined over $\F_q$.
Write $N=p^en_1n_2$ with $p\nmid n_1 n_2$ and $n_1\mid n_2$
$($possibly $n_1=1)$. Then, there is an elliptic curve $E$ over $\F_q$ for
which
$$
E(\F_q)\cong \Z_{p^e}\times\Z_{n_1}\times\Z_{n_2}
$$
if and only if
\begin{enumerate}
\item[$(i)$] $n_1=n_2$ in  case~$(ii)$ of Lemma~\ref{lem:Water};
\item[$(ii)$] $n_1\mid q-1$ in all  other cases of Lemma~\ref{lem:Water}.
\end{enumerate}
\end{lemma}

Combining Lemmas~\ref{lem:Water} and~\ref{lem:Ruck}, we get:

\begin{corollary}
\label{cor:WR} If $p$ is prime and $N\in\N$ with $|p+1-N|\le 2\sqrt{p}$,
then there is an elliptic curve $E$ defined over $\F_p$ with $\#E(\F_p)=N$.
In this case, if we write $N=n_1n_3$ with $p\nmid n_1$ and $n_1\mid n_3$
$($possibly $n_1=1)$, then $n_1\mid p-1$ and
$E(\F_p)\cong\Z_{n_1}\times\Z_{n_3}$.
\end{corollary}

\begin{lemma}
\label{lem:Jp-struct}
A prime $p$ lies in $\cP(n,k)$ if and only if
$p=kn^2+\ell n+1$ for some integer $\ell$ such that $|\ell|\le 2\sqrt{k}$.
\end{lemma}

\begin{proof} By definition, if $p$ lies in $\cP(n,k)$ then there is
an elliptic curve $E/\F_p$ such that
$E(\F_p)\cong\Z_n\times\Z_{kn}$. According to Lemma~\ref{lem:character},
$p=kn^2+\ell n+1$ with some integer $\ell$ such that
$|\ell|\le 2\sqrt{k}$.

Conversely, suppose that $p=kn^2+\ell n+1$ and $|\ell|\le
2\sqrt{k}$. Taking $N=kn^2$ we have
$$
|p+1-N|^2=(\ell n+2)^2=\ell^2n^2+4\ell n+4\le 4kn^2+4\ell n+4=4p,
$$
hence $|p+1-N|\le 2\sqrt{p}$.  Applying Corollary~\ref{cor:WR} with
$n_1=n$ and $n_3=kn$, we see that there is
an elliptic curve $E/\F_p$ such that
$E(\F_p)\cong\Z_n\times\Z_{kn}$, and thus $p\in\cP(n,k)$.
\end{proof}

Next, we relate $\NP(N,K)$ to the distribution of primes
in short arithmetic progressions.

\begin{lemma}
\label{lem:Sp-struct} For all $N,K\in\N$ we have
$$
\NP(N,K)=\sum_{\substack{n\le N\\|\ell|\le 2\sqrt{K}}}
\Bigl(\pi(Kn^2+\ell n+1;n^2,\ell n+1)-\pi(\tfrac14 \ell^2n^2+\ell n+1;n^2,\ell n+1)\Bigr).
$$
\end{lemma}

\begin{proof} Fix $n\le N$, and let
$\cT_1(n)$ be the collection of pairs $(\ell,p)$ such that
\begin{itemize}
\item[$(i)$] $|\ell|\le 2\sqrt{K}$;
\item[$(ii)$] $p$ is a prime congruent to $\ell n+1\pmod{n^2}$;
\item[$(iii)$] $\tfrac14 \ell^2n^2+\ell n+1\le p\le Kn^2+\ell n+1$.
\end{itemize}
Since $\tfrac14 \ell^2n^2+\ell n+1=(\tfrac12\ell n+1)^2$ cannot be prime, it is
easy to see that
$$
\#\cT_1(n)=\sum_{|\ell|\le 2\sqrt{K}}
\Bigl(\pi(Kn^2+\ell n+1;n^2,\ell n+1)-\pi(\tfrac14 \ell^2n^2+\ell n+1;n^2,\ell n+1)\Bigr).
$$
Let $\cT_2(n)$ be the collection of pairs $(k,p)$ such that
\begin{itemize}
\item[$(iv)$] $k\le K$;
\item[$(v)$] $p$ is prime and $p=kn^2+\ell n+1$ for some integer $\ell$
such that $|\ell|\le 2\sqrt{k}$.
\end{itemize}
By Lemma~\ref{lem:Jp-struct}, condition $(v)$ is equivalent to the assertion
that $p\in\cP(n,k)$, hence
$$
\#\cT_2(n)=\sum_{k\le K}\#\cP(n,k).
$$

Since
$$
\sum_{n\le N}\#\cT_1(n)=\sum_{\substack{n\le N\\|\ell|\le 2\sqrt{K}}}
\Bigl(\pi(Kn^2+\ell n+1;n^2,\ell n+1)-\pi(\tfrac14 \ell^2n^2+\ell n;n^2,\ell n+1)\Bigr)
$$
and
$$
\sum_{n\le N}\#\cT_2(n)=\sum_{n\le N}\sum_{k\le K}\#\cP(n,k)=\NP(N,K),
$$
to prove the lemma it suffices to show
that $\#\cT_1(n)=\#\cT_2(n)$ for each~$n\le N$.

First, let $(\ell,p)\in\cT_1(n)$.  By $(ii)$ we can write $p=kn^2+\ell n+1$ for some
integer $k$.  Substituting into $(iii)$ we have
$$
\tfrac14 \ell^2n^2+\ell n+1\le kn^2+\ell n+1\le Kn^2+\ell n+1,
$$
hence $k\le K$ and $|\ell|\le 2\sqrt{k}$.  This shows that the pair $(k,p)$
lies in $\cT_2(n)$.  As the map $\cT_1(n)\to\cT_2(n)$ given by
$(\ell,p)\mapsto (k,p)$ is clearly injective, we have $\#\cT_1(n)\le\#\cT_2(n)$.

Next, suppose that $(k,p)\in\cT_2(n)$, and let $\ell$ be as in $(v)$.
By $(iv)$ we have $|\ell|\le 2\sqrt{k}\le 2\sqrt{K}$, and $p\equiv \ell n+1\pmod{n^2}$
by $(v)$.  Furthermore, since $\tfrac14\ell^2\le k\le K$ the prime $p=kn^2+\ell n+1$
satisfies $(iii)$.  This shows that the pair $(\ell,p)$
lies in $\cT_1(n)$.  Since the map $\cT_2(n)\to\cT_1(n)$ given by
$(k,p)\mapsto (\ell,p)$ is injective, we have $\#\cT_2(n)\le\#\cT_1(n)$,
and the proof is complete.
\end{proof}

\section{Primes in sparse progressions}

Below, we use the following result of Baier and Zhao~\cite{BZ},
which is a variant of the Bombieri-Vinogradov theorem that deals with
primes in arithmetic progressions to square moduli.

\begin{lemma}
\label{lem:BZ}
For fixed $\eps>0$ and $C>0$ we have 
$$
\sum_{m\le x^{2/9-\eps}}m~\max_{\gcd(a,m)=1}
\left|\psi(x;m^2,a)-\frac{x}{\varphi(m^2)}\right|\ll
\frac{x}{(\log x)^C}\,,
$$
where the implied constant depends only on $\eps$ and $C$.
\end{lemma}

Via partial summation one obtains the following:

\begin{corollary}
\label{cor:BZ}
For fixed $\eps>0$ and $C>0$ we have 
$$
\sum_{m\le x^{2/9-\eps}}m~\max_{\gcd(a,m)=1}
\left|\pi(x;m^2,a)-\frac{\pi(x)}{\varphi(m^2)}\right|\ll
\frac{x}{(\log x)^C}\,,
$$
where the implied constant depends only on $\eps$ and $C$.
\end{corollary}

For our applications of Corollary~\ref{cor:BZ}
we also need a well known asymptotic formula
\begin{equation}
\label{eq:sum n/phi(n)}
\sum_{n\le X}\frac{n}{\varphi(n)}=\frac{315\,\zeta(3)}{2\pi^4}\,X+O(\log X);
\end{equation}
for more precise results, we refer the reader to~\cite{Nowak,Sita1,Sita2}.
 
For any  sequence of 
integers $\cA = (a_n)_{n =1}^\infty$  
and any positive real numbers $\lambda$ and $X$, we define the sum
\begin{equation}
\label{eq:sum SA}
\sP(\cA;\lambda,X) = 
\sum_{ n\le X}\pi(\lambda n^2;n^2,a_n).
\end{equation}

\begin{lemma}
\label{lem:SA-sum}
Fix $\eps\in(0,2/5)$.  For any sequence of integers $\cA=(a_n)_{n=1}^\infty$ 
such that $\gcd(a_n,n)=1$ for all $n$, and for any real numbers
$\lambda$ and $X$ such that $3\le X\le\lambda^{2/5-\eps}$,
the estimate
$$
\sP(\cA;\lambda,X)=\frac{315\,\zeta(3)}{2\pi^4}
\frac{\lambda X}{\log(\lambda X^2)}+O\(\frac{\lambda X(\log\log X)^2}{(\log X)^{2}}\)
$$
holds, where the implied constant depends only on $\eps$. 
\end{lemma}

\begin{proof} Let $\Delta$ be an arbitrary real number such
that $X^{-1}\le\Delta\le 1$, and let
$$
J=\fl{\frac{2\log\log X}{\log(1+\Delta)}}
\ll\Delta^{-1}\log\log X.
$$
Put 
$$
X_j=X(1+\Delta)^{j-J} \qquad (j=0,1,\ldots,J). 
$$
Note that
$$
\frac{X}{(\log X)^{2}}\le X_0\le\frac{2X}{(\log X)^2}\,,$$
and we have
$$
X_j\le X_{j+1}\le 2X_j\mand\log X_j \gg \log X.
$$

Using the trivial bound $\pi(\lambda n^2;n^2,a_n)\le \lambda$ for all $n\le X_0$,
we derive that
\begin{equation}
\label{eq:Split-j}
\begin{split}
\sP(\cA;\lambda,X) 
& = \sum_{X_0 <  n\le X} \pi(\lambda n^2;n^2,a_n) + O(\lambda X_0)\\
& = \sum_{j=0}^{J-1} S_j + O\( \frac{\lambda X}{(\log X)^{2}}\), 
\end{split}
\end{equation}
where
$$
S_j =  \sum_{X_j<  n\le X_{j+1}} \pi(\lambda n^2;n^2,a_n)
 \qquad (j=0,1,\ldots,J). 
$$
Since $X_{j+1}-X_j=\Delta X_j$, for every integer $n \in [X_j, X_{j+1}]$ we have 
\begin{equation}
\label{eq:Approx}
n^2 = X_j^2 + O(\Delta X_j^2).
\end{equation}
For any such $n$, the number of primes $p \in [\lambda X_j^2, \lambda n^2]$ 
with $p \equiv a_n\pmod{n^2}$ does not exceed
$$
\frac{\lambda n^2 - \lambda X_j^2}{n^2} + 1 \ll 
\frac{\Delta\lambda X_j^2}{n^2} + 1 
\le \Delta\lambda + 1 \ll \Delta\lambda
$$
(since $\Delta\lambda\ge \Delta X\ge 1$). Therefore, 
\begin{equation}
\label{eq:Trunc}
S_j =  \sum_{X_j<  n\le X_{j+1}} \pi(\lambda X_j^2;n^2,a_n) + 
O(\Delta^2\lambda X_{j})\qquad (j=0,1,\ldots,J). 
\end{equation}
Furthermore,  
\begin{eqnarray*}
\lefteqn{
\left| \sum_{X_j<  n\le X_{j+1}} \pi(\lambda X_j^2;n^2,a_n)  - \pi(\lambda X_j^2) 
\sum_{X_j<  n\le X_{j+1}} \frac{1}{\varphi(n^2)}\right|}\\
& & \qquad\qquad \le \sum_{X_j<  n\le X_{j+1}} 
\left| \pi(\lambda  X_j^2;n^2,a_n)  -  \frac{\pi(\lambda X_j^2) }{\varphi(n^2)}\right|\\
& & \qquad \qquad \le \frac{1}{X_j} \sum_{X_j<  n\le X_{j+1}} n
\left| \pi(\lambda  X_j^2;n^2,a_n)  -  \frac{\pi(\lambda X_j^2) }{\varphi(n^2)}\right|\\
& & \qquad \qquad\le \frac{1}{X_j} \sum_{n\le X_{j+1}} n~\max_{\gcd(a,n)=1}
\left| \pi(\lambda X_j^2;n^2,a)-\frac{\pi(\lambda X_j^2) }{\varphi(n^2)}\right|.
\end{eqnarray*}
In view of the hypothesis that $3\le X\le\lambda^{2/5-\eps}$
we can apply Corollary~\ref{cor:BZ} with $C=4$ to derive the bound
\begin{equation}
\label{eq:Two sums}
\sum_{X_j<  n\le X_{j+1}} 
\pi(\lambda X_j^2;n^2,a_n)  - \pi(\lambda X_j^2) \sum_{X_j<  n\le X_{j+1}}
\frac{1}{\varphi(n^2)} \ll  \frac{\lambda X_j }{(\log X)^4}\,.
\end{equation}
Using~\eqref{eq:Approx} again, we write
$$
\pi(\lambda X_j^2) \sum_{X_j<  n\le X_{j+1}} \frac{1}{\varphi(n^2)}
= \sum_{X_j<  n\le X_{j+1}} 
\frac{\pi(\lambda n^2) + O(\Delta\lambda X_j^2)}{\varphi(n^2)}\,.
$$
Using the prime number theorem in its simplest form, namely
$$
\pi(y)=\frac{y}{\log y}+O\(\frac{y}{(\log y)^2}\),
$$
(see~\cite[Chapter~II.4, Theorem~1]{Ten} for a 
stronger statement) together with the lower bound
$$
\varphi(n^2)=n\varphi(n)\gg\frac{n^2}{\log\log(n+2)}\qquad(n\in\N)
$$
(see~\cite[Chapter~I.5, Theorem~4]{Ten}) and the trivial inequalities
$$
\log(\lambda X^2) \ge \log(\lambda n^2)
\ge \log(\lambda X_0^2) = \log(\lambda X^2) + O(\log \log X),
$$
which hold for any integer $n\in[X_0,X]$, we derive that
\begin{equation*}
\begin{split}
&\pi(\lambda X_j^2) \sum_{X_j<  n\le X_{j+1}} \frac{1}{\varphi(n^2)}\\
&~~= \lambda \sum_{X_j<  n\le X_{j+1}} 
\frac{n^2}{\varphi(n^2)\log(\lambda n^2)}
+O\(\frac{\Delta\lambda X_j\log\log X}{(\log X)^2}
+ \Delta^2 \lambda X_j\log \log X\) \\
&~~=\frac{\lambda}{\log(\lambda X^2)} \sum_{X_j<  n\le X_{j+1}} 
\frac{n}{\varphi(n)}
+O\( \frac{\Delta\lambda X_j(\log \log X)^2}{(\log X)^2} 
+  \Delta^2\lambda X_{j} \log \log X \).
\end{split}
\end{equation*}
Combining this result with~\eqref{eq:Trunc} and~\eqref{eq:Two sums} we see that
\begin{equation*}
\begin{split}
&S_j-\frac{\lambda}{\log(\lambda X^2)} \sum_{X_j<n\le X_{j+1}} 
\frac{n}{\varphi(n)}\\
&\qquad\qquad \ll  \frac{\lambda X_j }{(\log X)^4} +
\frac{\Delta\lambda X_j (\log \log X)^2}{ (\log X)^2}
+ \Delta^2 \lambda X_{j}\log \log X.
\end{split}
\end{equation*}
We insert this estimate in~\eqref{eq:Split-j} and deduce that
\begin{equation*}
\begin{split}
&\sP(\cA;\lambda,X)-\frac{\lambda}{\log(\lambda X^2)}\sum_{X_0<n\le X} 
\frac{n}{\varphi(n)}\\
&\qquad\qquad  \ll   \(\frac{\lambda }{(\log X)^4} +
\frac{\Delta\lambda(\log \log X)^2}{ (\log X)^2} 
+\Delta^2\lambda\log \log X\)  \sum_{j=0}^{J-1} X_{j}\\
&\qquad\qquad  \ll  \(\frac{\lambda }{(\log X)^4} +
\frac{\Delta\lambda(\log \log X)^2}{ (\log X)^2} 
+\Delta^2\lambda  \log \log X\)   \Delta^{-1} X\\
&\qquad\qquad  =  \frac{\Delta^{-1}\lambda X}{(\log X)^4} +
\frac{\lambda X(\log \log X)^2}{ (\log X)^2} 
+\Delta \lambda X \log \log X.
\end{split}
\end{equation*}
Taking $\Delta = (\log X)^{-2}$ (for which our hypothesis
$X^{-1}\le\Delta\le 1$ holds for all $X>1$) and taking into account
that~\eqref{eq:sum n/phi(n)} implies the estimate
\begin{equation*}
\begin{split}
\sum_{X_0<n\le X}\frac{n}{\varphi(n)}&=
\frac{315\,\zeta(3)}{2\pi^4}\,(X-X_0)+O(\log X)\\
&=\frac{315\,\zeta(3)}{2\pi^4}\,X+O\(\frac{X}{(\log X)^{2}}\), 
\end{split}
\end{equation*}
we conclude the proof. 
\end{proof}

We are certain that the error term of Lemma~\ref{lem:SA-sum} can be improved
easily, but we have not attempted to do so as we only require
the asymptotic behavior of $\sP(\cA;\lambda,X)$ stated in the next
corollary.

\begin{corollary}
\label{cor:SA-sum}
Fix $\eps\in(0,2/5)$.  For any sequence of integers $\cA=(a_n)_{n=1}^\infty$ 
such that $\gcd(a_n,n)=1$ for all $n$, and for any real numbers
$\lambda$ and $X$ such that $\lambda^\eps\le X\le\lambda^{2/5-\eps}$,
the estimate
$$
\sP(\cA;\lambda,X)=\(\frac{315\,\zeta(3)}{2\pi^4}+o(1)\)\,  
\frac{\lambda X}{\log(\lambda X^2)}
$$
holds, where the function implied by $o(1)$ depends only on $\eps$. 
\end{corollary}

\section{The sets $\Sp(N,K)$ and $\Sq(N,K)$}
\label{sec:SpSq}

We begin with the observation that
\begin{equation}
\label{hell0}
\#\Sp(N,K)\ge \sum_{n\le N}\pi(Kn^2;n^2,1).
\end{equation}
Indeed, if $p=kn^2+1$ is a prime which does not exceed $Kn^2$, then the pair
$(n,(p-1)/n^2)$ lies in $\Sp(N,K)$.  Clearly, Corollary~\ref{cor:SA-sum}
can be applied to the sum on the right hand side of~\eqref{hell0} to
derive the lower bound
$$
\#\Sp(N,K)\ge\(\frac{315\,\zeta(3)}{2\pi^4}+o(1)\)\frac{KN}{\log(KN^2)}
$$
provided that $K^\eps\le N\le K^{2/5-\eps}$.
Moreover, even without the condition $N\ge K^{\eps}$
we are able to get a lower bound of the same strength.  

\begin{theorem}
Fix $\eps\in(0,2/5)$, and suppose that
$N\le K^{2/5-\eps}$. Then, the following bound holds:
$$
\#\Sp(N,K) \gg \frac{KN}{\log K}\,.
$$
\end{theorem}

\begin{proof} Using~\eqref{hell0} together with the elementary bound
$$
\frac{\psi(x;m,a)}{\log x}\le\Pi(x;m,a)=\pi(x;m,a)+O\(x^{1/2}\log x\),
$$
we have
\begin{equation*}
\begin{split}
\#\Sp(N,K)&\ge\sum_{N/2\le n\le N}\pi(Kn^2;n^2,1)\\
&\ge\sum_{N/2\le n\le N}\(\frac{\psi(Kn^2;n^2,1)}{\log(Kn^2)}+O\(K^{1/2}n\log(Kn^2)\)\)\\
&\gg \frac{1}{\log K}\sum_{N/2\le n\le N}\psi(\tfrac14 KN^2;n^2,1)+O\(K^{1/2}N^2\log K\)\\
&=\frac1{\log K}\sum_{N/2\le n\le N}\frac{KN^2}{4\varphi(n^2)}+E(N,K)
+O\(K^{1/2}N^2\log K\),
\end{split}
\end{equation*}
where
\begin{equation*}
\begin{split}
\bigl|E(N,K)\bigr|&\le\frac{1}{\log K}\sum_{N/2\le n\le N}\left|\psi(\tfrac14 KN^2;n^2,1)
-\frac{ KN^2}{4\varphi(n^2)}\right|\\
&\le \frac{2}{N\log K}\sum_{N/2\le n\le N}n\left|\psi(\tfrac14 KN^2;n^2,1)
-\frac{\ KN^2}{4\varphi(n^2)}\right|.
\end{split}
\end{equation*}
Applying Lemma~\ref{lem:BZ} with $x=\tfrac14 KN^2$ and $C=1$
(which is permissible since our assumption $N\le K^{2/5-\eps}$ implies
that $N\le(\tfrac14 KN^2)^{2/9-\delta}$ for a suitable $\delta>0$ that depends
only on $\eps$) we see that
$$
E(N,K)\ll \frac{KN}{(\log K)^2}\,,
$$
and therefore,
$$
\#\Sp(N,K)\gg \frac{KN^2}{\log K}\sum_{N/2\le n\le N}\frac1{\varphi(n^2)}+
O\(K^{1/2}N^2\log K+\frac{KN}{(\log K)^2}\).
$$
Since
$$
\sum_{N/2\le n\le N}\frac1{\varphi(n^2)}
\ge \sum_{N/2\le n\le N}\frac1{n^2}\gg \frac1{N}\,,
$$
the result follows.
\end{proof}

\begin{theorem}
\label{thm:Sp-bound}
For any fixed $K\in\N$ we have
$$
\#\Sp(N,K)\ll_K \frac{N}{\log N}\,.
$$
\end{theorem}

\begin{proof} The Selberg sieve provides the following
upper bound on the number of primes represented by an irreducible polynomial
$F(n)=an^2+bn+1$ with integer coefficients 
(see Halberstam and Richert~\cite[Theorem~5.3]{HR}
for a more general statement):
\begin{equation}
\label{eq:SelbergSieve}
\begin{split}
\#\bigl\{n\le x~:~F(n)\text{~is prime}\bigr\}&\le 2\prod_p
\(1-\frac{\chi_p(b^2-4a)}{p-1}\) \\
&\qquad\times\quad\frac{x}{\log x}\(1+O_F\(\frac{\log\log 3x}{\log x}\)\),
\end{split}
\end{equation}
where $\chi_p$ is the quadratic character modulo~$p$, that is, the Dirichlet
character afforded by the Legendre symbol. The constant implied by
$O_F$ depends on~$F$, and this is the reason that $K$ is fixed in the
statement of the theorem.

Trivially, we have
$$
\#\Sp(N,K) \le
\sum_{k\le K}\sum_{|\ell|<2\sqrt{k}}\#\bigl\{n\le N: kn^2+\ell n+1\text{~is prime}\bigr\}.
$$
Applying~\eqref{eq:SelbergSieve} with $F(n)=kn^2+\ell n+1$,
the result is immediate.
\end{proof}

\begin{corollary}
\label{cor:Sq-bound}
For any fixed $K\in\N$ we have
$$
\#\Sq(N,K)\ll_K \frac{N}{\log N}\,.
$$
\end{corollary}

\begin{proof}
We have
\begin{equation}
\label{eq:franc}
\#\Sq(N,K)\le \#\Sp(N,K)+\sum_{j=2}^\infty\#\Sq^{(j)}(N,K),
\end{equation}
where for each $j\ge 2$, we use $\Sq^{(j)}(N,K)$ to denote the set of pairs
$(n,k)$ in $\Sq(N,K)$ associated with prime
powers of the form $q=p^j$ with $p$ prime. It is easy to see that
$$
\#\Sq^{(j)}(N,K)\ll K^{3/2} \pi\bigl(\bigl(KN^2+2K^{1/2}N+1\bigr)^{1/j}\bigr)
\ll\begin{cases}
K^{2}N/\log N&\ \text{if $j=2$,}\\
K^{11/6}N^{2/3}&\ \text{if $j\ge 3$.}\\
\end{cases}
$$
Indeed, for fixed $k$ and $p$ there are only $O(K^{1/2})$ 
possibilities for $\ell$. Thus, for fixed $p$ there are $O(K^{3/2})$ possibilities
for $(n,k)$, where the implied constant is absolute. 
Furthermore, $\Sq^{(j)}(N,K)=\varnothing$ for all but $O\(\log(KN)\)$ choices
of~$j$. Thus, from~\eqref{eq:franc} we deduce that
$$
\# \Sq(N,K)\le\#\Sp(N,K)+O_K(N/\log N),
$$
and the result follows from Theorem~\ref{thm:Sp-bound}.
\end{proof}

An immediate consequence of Corollary~\ref{cor:Sq-bound} is
that there are infinitely many pairs $(n,k)$ that do not lie in $\Sq$.
In fact, if $k\in\N$ is fixed, then we see that $(n,k)\not\in\Sq$ for
almost all $n\in\N$.

The situation is very different when $n\in\N$ is fixed,
for in this case we expect that the pair $(n,k)$ lies in the
smaller set $\Sp$ for all but finitely many $k\in\N$.  To prove
this, one needs to show that 
$$
\pi((k^{1/2}n+1)^2;n,1)-\pi((k^{1/2}n-1)^2;n,1)>0
$$
for all sufficiently large $k$. Although this problem
is intractable at present, the probabilistic model of Cram\'er
(see, for example, \cite{Gran,Sound}) predicts that
$$
\pi((k^{1/2}n+1)^2;n,1)-\pi((k^{1/2}n-1)^2;n,1)\gg_n k^{1/2}/\log k
$$ 
for all large $k$.  Unconditionally, it may be possible to
answer the following questions:

\begin{itemize}
\item If $n\in\N$ is fixed, is it true that $(n,k)\in\Sq$ for almost all $k\in\N$?

\item Is it true that for almost all $n\in\N$, there are only finitely many
pairs $(n,k)$ that do not lie in $\Sq$?
\end{itemize}

We conclude this section with the following:

\begin{theorem}
\label{thm:adam}
The set $\Sq\setminus\Sp$ is infinite.  In fact, we have
$$
\#\bigl\{n\in N~:~(n,1)\in\Sq\setminus\Sp\bigl\}
\ge\(2+o(1)\)\frac{N}{\log N}\qquad(N\to\infty).
$$
\end{theorem}

\begin{proof}
Using the prime number theorem for arithmetic progressions
together with a standard upper bound from
sieve theory such as~\cite[Theorem~5.3]{HR}, one sees
that there are $(2+o(1))N/\log N$ natural numbers $n\le N$
such either $n-1$ or $n+1$ is prime, but not both, and such that
the integers $n^2+1$, $n^2+n+1$ and $n^2-n+1$ are all composite.
For any such $n$, either $(n-1)^2$ or $(n+1)^2$ is a prime power,
and we have $(n,1)\in\Sq$; however, $n^2+\ell n+1$ is clearly
composite for $-2\le\ell\le 2$, and thus $(n,1)\not\in\Sp$.
\end{proof}

\section{The double sum $\NP(N,K)$}

Here, we study the double sum $\NP(N,K)$ using the formula of Lemma~\ref{lem:Sp-struct}.
Our main result is the following:

\begin{theorem}
\label{eq:Sum N}
Fix $\eps\in(0,2/5)$, and suppose that 
$K^{\eps}\le N\le K^{2/5-\eps}$. Then, the estimate
$$
\NP(N,K)=\(\frac{210\,\zeta(3)}{\pi^4}+o(1)\)\frac{K^{3/2}N}{\log(KN^2)}
$$
holds, where the function implied by $o(1)$ depends only on $\eps$. 
\end{theorem}

\begin{proof}
Using the trivial estimate
$$
\pi(x+y;k,a) = \pi(x;k,a) + O(y/k+1),
$$
we see from Lemma~\ref{lem:Sp-struct} that $\NP(N,K)$ is equal to
\begin{equation*}
\begin{split}
\sum_{\substack{n\le N\\|\ell|\le 2\sqrt{K}}}&
\Bigl(\pi(Kn^2;n^2,\ell n+1)-\pi(\tfrac14 \ell^2n^2;n^2,\ell n+1) + O(\ell/n + 1)\Bigr)\\
&=\sum_{\substack{n\le N\\|\ell|\le 2\sqrt{K}}}
\Bigl(\pi(Kn^2;n^2,\ell n+1)-
\pi(\tfrac14 \ell^2n^2;n^2,\ell n+1)\Bigr)\\
& \qquad \qquad \qquad \qquad \qquad \qquad \qquad \qquad +O(K\log N+K^{1/2}N)\\
&=\sum_{|\ell|\le 2\sqrt{K}} \(\sP(\cA_\ell;K,N) - 
\sP(\cA_\ell;\tfrac14 \ell^2,N)\)+ O(K \log N), 
\end{split}
\end{equation*}
where $\cA_\ell=\(n \ell + 1\)_{n=1}^\infty$ for each $\ell$,
and the sum $\sP(\cA_\ell;\lambda,X)$ is defined by~\eqref{eq:sum SA}.
Note that we have used the bound $K^{1/2}N\ll K\log N$, which follows
from our hypothesis that $N\le K^{2/5-\eps}$.

We now put $L=2\sqrt{K}/\log K$ and write 
\begin{equation}
\label{eq:NS1S2}
\NP(N,K)  =  S_1 + S_2+ O(K \log N),
\end{equation}
where 
\begin{equation*}
\begin{split}
S_1&=\sum_{|\ell|\le L} \(\sP(\cA_\ell;K,N) - 
 \sP(\cA_\ell;\tfrac14 \ell^2,N)\),\\
S_2&=\sum_{L < |\ell| \le 2\sqrt{K}} \(\sP(\cA_\ell;K,N) - 
 \sP(\cA_\ell;\tfrac14 \ell^2,N)\).
\end{split}
\end{equation*}

For $S_1$ we use the trivial estimate 
$$
S_1\le\sum_{|\ell|\le L}\sP(\cA_\ell;K,N)
$$
together with Corollary~\ref{cor:SA-sum} to derive the bound 
\begin{equation}
\label{eq:S1 bound}
 S_1 \ll \frac{L KN}{\log K} \ll \frac{K^{3/2}N}{(\log K)^2}\,.
\end{equation}

For $S_2$ we apply Corollary~\ref{cor:SA-sum} to both terms 
in the summation.  Writing $\Theta=315\,\zeta(3)/(2\pi^4)$, and
taking into account that
$$
\log(\ell^2 N^2/4)=(1+o(1))\log(KN^2)\qquad(L<|\ell|\le 2\sqrt{K}),
$$
we see that
\begin{equation*}
\begin{split}
S_2&=\sum_{L < |\ell|\le 2\sqrt{K}}
\((\Theta+o(1)) \frac{KN}{\log(KN^2)} - 
(\Theta+o(1)) \frac{\ell^2N}{4 \log(\ell^2N^2/4)}\)\\
& =(\Theta+o(1)) \frac{N}{\log(KN^2)}  \sum_{L < |\ell|\le 2\sqrt{K}}
(K -   \ell^2/4)
=\big(\tfrac43\Theta+o(1)\big)\frac{K^{3/2}N}{\log(KN^2)}\,. 
\end{split}
\end{equation*}
Using this bound and~\eqref{eq:S1 bound} in~\eqref{eq:NS1S2}, we finish the proof. 
\end{proof}

\section{The sets $\cN_{m,k}$ and $\widetilde\cN_{m,k}$}

In this section, we study the sets $\cN_{m,k}$ and $\widetilde\cN_{m,k}$
introduced in \S\ref{sec:intro}. We begin the following:

\begin{lemma}
\label{lem:Nkm incl} For all $m,k\in\N$ we have
$\cN_{m,k}\subseteq\widetilde\cN_{m,k}$.
\end{lemma}

\begin{proof} For every $n\in\cN_{m,k}$ there is a prime $p$ and
an elliptic curve $E$ defined over $\F_{p^m}$ such that
$E(\F_{p^m})\cong \Z_n\times\Z_{kn}$.  By Lemma~\ref{lem:character},
$p^m=kn^2+\ell n+1$ for some integer $\ell$ that satisfies
$|\ell|\le 2\sqrt{k}$, that is, $n\in\widetilde\cN_{m,k}$.
\end{proof}

\subsection{Results with fixed values of $m$}

In the case that $m=1$, the set inclusion of
Lemma~\ref{lem:Nkm incl} is an equality. 

\begin{theorem}
\label{thm:Nk1} For all $k\in\N$ we have
$\cN_{1,k}=\widetilde\cN_{1,k}$. 
\end{theorem}

\begin{proof}
In view of Lemma~\ref{lem:Nkm incl} it suffices to show that
$\widetilde\cN_{1,k}\subseteq\cN_{1,k}$.  For every $n\in\widetilde\cN_{1,k}$
there is a prime $p$ such that $p=kn^2+\ell n+1$.
Put $a=n\ell+2$, and note that $|a|\le 2\sqrt{p}$ since
$$
a^2=n^2\ell^2+4n\ell+4\le 4(n^2k+n\ell+1)=4p.
$$

If $\gcd(a,p)=1$, then by Lemma~\ref{lem:Water}$\,(i)$
there is an elliptic curve $E/\F_p$ such that
$\#E(\F_p)=p+1-a=kn^2$. On the other hand, if $p\mid a$,
then the inequality $|a|\le 2\sqrt{p}$ implies that
either $p\le 3$ and $a=\pm p$, or $a=0$.  Applying
Lemma~\ref{lem:Water}$\,(iv)$ in the former case and
Lemma~\ref{lem:Water}$\,(iv)$ in the latter, we again
conclude that there is an elliptic curve $E/\F_p$ such that
$\#E(\F_p)=kn^2$.
In all cases, since $p\equiv 1\pmod n$, Lemma~\ref{lem:Ruck}$\,(ii)$
guarantees that there is an elliptic curve $E$ defined
over $\F_p$ such that $E(\F_p)\cong \Z_n\times\Z_{kn}$.
Therefore, $n\in\cN_{1,k}$.
\end{proof}

\begin{lemma}
\label{lem:ryan}
For natural numbers $n,k$ the set
$$
\widetilde\cP(n,k)=\bigl\{\text{primes~}p~:~p^2=kn^2+\ell n+1
\text{~for some~}\ell\in\Z\text{~with~}|\ell|\le 2\sqrt{k}\,\bigr\}
$$
contains at most one prime except for the following cases:
\begin{itemize}
\item[$(i)$] $\widetilde\cP(n,k)=\{2,3\}$ if $n=1$ and $4\le k\le 9$;

\item[$(ii)$] $\widetilde\cP(n,k)=\{hn\pm 1\}$ if $k=h^2$ for some $h\in\N$,
and both $hn-1$ and $hn+1$ are primes.
\end{itemize}
\end{lemma}

\begin{proof}
It is easy to see that
\begin{equation}
\label{eq:tPchar}
\widetilde\cP(n,k)=\bigl\{\text{primes~}p\in
\big[n\sqrt{k}-1,n\sqrt{k}+1\big]~:~p^2\equiv 1\pmod n\bigr\}.
\end{equation}
Since the interval
$\big[n\sqrt{k}-1,n\sqrt{k}+1\big]$
has length two, the result follows immediately.
\end{proof}

When $m=2$, the inclusion of Lemma~\ref{lem:Nkm incl} can be proper.
Fortunately, we are able to classify those natural numbers $k$ for which this happens.

\begin{theorem}
\label{thm:N2k}
For all $k\in\N$ we have $\cN_{2,k}=\widetilde\cN_{2,k}$
except for the following disjoint cases:
\begin{itemize}
\item[$(i)$] $k=p^2+1$ for some prime $p\equiv1\pmod4$;
\item[$(ii)$] $k=p^2\pm p+1$ for some prime $p\equiv1\pmod3$; 
\item[$(iii)$] $k=h^2$ for some integer $h>1$.
\end{itemize}
In cases~$(i)$ and~$(ii)$ we have $\widetilde\cN_{2,k}\setminus \cN_{2,k}=\{1\}$,
and in case~$(iii)$ we have
\begin{equation}
\label{eq:monster-truck-nitro}
\widetilde\cN_{2,k}\setminus\cN_{2,k}=
\bigr\{n\in\N~:~hn-1\text{~or~}hn+1\text{~is prime}\bigl\}.
\end{equation}
\end{theorem}

\begin{proof} Let $k$ be fixed, and suppose that $n\in\widetilde\cN_{2,k}$.
Let $p$ and $\ell$ be such that $p^2=kn^2+\ell n+1$, $|\ell|\le 2\sqrt{k}$,
and put $a=\ell n+2$. Then $|a|\le 2p$, and using
Lemmas~\ref{lem:Water} and~\ref{lem:Ruck} it is easy to see that
$n$ lies in $\cN_{2,k}$ except possibly in the following cases:
\begin{itemize}
\item[(1)] $a=0$ and $p\equiv 1\pmod 4$;

\item[(2)] $a=\pm p$ and $p\equiv 1\pmod 3$;

\item[(3)] $a=\pm 2p$ and $k$ is not of the form $p^j$ for any $j\ge 0$.
\end{itemize}

In case (1) we have $\ell n=-2$, which implies either that
$(n,\ell)=(2,-1)$ and $p^2=4k-1$, which is impossible, or
that $(n,\ell)=(1,-2)$ and $p^2=k-1$. This shows that
$\widetilde\cN_{2,k}\setminus \cN_{2,k}\subseteq\{1\}$
and that $k$ satisfies the condition $(i)$.
Since $k\ge 26$ and $k\ne h^2$ for any $h>1$, we have $\widetilde\cP(n,k)=\{p\}$
by Lemma~\ref{lem:ryan}.
It remains to show that $1\not\in\cN_{2,k}$ in this case.
Suppose on the contrary that $1\in\cN_{2,k}$.  Then there is a prime $p_0$ and
an elliptic curve $E$ defined over $\F_{p_0^2}$ such that
$E(\F_{p_0^2})\cong\Z_1\times\Z_k$.  By Lemma~\ref{lem:character}
we see that $p_0^2=k+\ell+1$ for some integer $\ell$ such that
$|\ell|\le 2\sqrt{k}$; that is, $p_0\in\widetilde\cP(n,k)$.  Therefore,
$p_0=p$, and $\#E(\F_{p^2})=k$.  But this is impossible by
Lemma~\ref{lem:Water}$\,(v)$ since $p\equiv 1\pmod 4$.

In case (2) we have $p=\pm(\ell n+2)\equiv\pm 2\pmod n$, thus
$p^2\equiv 4\pmod n$. Since $p^2=kn^2+\ell n+1\equiv 1\pmod n$
as well, it follows that $n\mid 3$.  We claim that $n\ne 3$.  Indeed, if $n=3$, then
$p^2=9k+3\ell+1=9k\pm p-1$, and therefore $p^2\mp p+1\equiv 0\pmod 9$.  But this is
impossible as neither $X^2+X+1$ nor $X^2-X+1$ has a root in $\Z_9$. If
$n=1$, then $p^2=k+\ell+1=k\pm p-1$. This shows that
$\widetilde\cN_{2,k}\setminus \cN_{2,k}\subseteq\{1\}$
and that $k$ satisfies the condition $(ii)$.  The proof that
$1\not\in\cN_{2,k}$ is similar to that of the preceding case.

In case (3) we have $p^2=kn^2\pm 2p-1$, or $kn^2=(p\mp 1)^2$;
it follows that $n\mid p\mp 1$, and $k=h^2$ with $h=(p\mp 1)/n$.
Since $k\ne p^0$, we see that $k$ satisfies the condition $(iii)$.
It remains to establish~\eqref{eq:monster-truck-nitro}.

Fix $h>1$, and suppose that $n\in\widetilde\cN_{2,h^2}$.
Then $\widetilde\cP(n,h^2)\ne\varnothing$, where by~\eqref{eq:tPchar}
we have
$$
\widetilde\cP(n,h^2)=\bigl\{\text{primes~}p\in[hn-1,hn+1]~:~p^2\equiv 1\pmod n\bigr\}.
$$
First, suppose $\widetilde\cP(n,h^2)$ contains a prime $p$ in the open interval
$(hn-1,hn+1)$.  Then, using Lemma~\ref{lem:ryan}, we deduce that
$\widetilde\cP(n,h^2)=\{p\}$, and thus case (3) does not occur for
any prime in $\widetilde\cP(n,h^2)$. Also, the cases (1) and (2) cannot
occur, for otherwise $k=h^2$ would satisfy $(i)$ or $(ii)$, respectively,
rather than $(iii)$. Consequently, $n\in\cN_{2,h^2}$ in this case.

Next, suppose $\widetilde\cP(n,h^2)$ does not contain a prime $p$ in the open
interval $(hn-1,hn+1)$. If $p\in\widetilde\cP(n,h^2)$, then $p=hn\pm 1$ for some
choice of the sign, and we have $p^2+1-h^2n^2=\pm2hn+2=\pm 2p$.  If there were
an elliptic curve $E$ defined over $\F_{p^2}$ such that
$E(\F_{p^2})\cong\Z_n\times\Z_{h^2n}$, then by Lemma~\ref{lem:Water}$\,(ii)$
and Lemma~\ref{lem:Ruck}$\,(i)$ it would follow that $n=h^2n$, which is impossible
since $h>1$.  This argument shows that $n\not\in\cN_{2,h^2}$ in this case.
\end{proof}

\begin{corollary} Suppose that $k$ is not a perfect square.
Then,
$$
\#\cN_{2,k}(T)\ll_k\log T.
$$
\end{corollary}

\begin{proof}
In view of Lemma~\ref{lem:Nkm incl}, it is enough to show that 
$\#\widetilde\cN_{2,k}(T)\ll_k\log T$.

Suppose that $n\in\widetilde\cN_{2,k}$ with $n\le T$.
Then there is a prime $p$ and an integer $\ell$ such
that $p^2=kn^2+\ell n+1$, $|\ell|\le 2\sqrt{k}$, and we have
\begin{equation}
\label{eq:emerald}
\max\{2kn+\ell,2p\}\ll_k T.
\end{equation}
Since
$$
(2kn+\ell)^2-k(2p)^2=\ell^2-4k,
$$
the pair $(2kn+\ell,2p)$ is a solution of the Pell equation
\begin{equation}
\label{eq:pell}
X^2-kY^2=\ell^2-4k.
\end{equation}
Note that $\ell^2-4k\ne 0$ since $k$ is not a perfect square. 
It is well known (and easy to verify) that every solution $(x,y)\in\Z^2$
to an equation such as~\eqref{eq:pell} has the form
$$
x+y\sqrt{k}=\big(x_0+y_0\sqrt{k}\,\big)\omega^t \qquad(t\in\Z),
$$
where $(x_0,y_0)$ is an arbitrary fixed solution,
and $\omega$ is a fixed unit in $\Q\big(\sqrt{k}\,\big)$;
therefore,
$$
t\ll_k\log\max\{|x|,|y|\}.
$$
In view of~\eqref{eq:emerald} we have $t\ll_k\log T$
for every solution $(x,y)=(2kn+\ell,2p)$ to~\eqref{eq:pell},
and the result follows.
\end{proof}

We remark that Theorem~\ref{thm:N2k} implies
$$
\#\cN_{2,1}(T)=\pi(T-1)+\pi(T+1)-\#\big\{p\le T-1: p+2\text{~is prime}\big\}
\sim\frac{2T}{\log T}\,.
$$

For $m\ge3$, the situation is more complicated.  For
example, it is easy to see that $3\in\widetilde\cN_{3,237}\setminus\cN_{3,237}$. 
Indeed, since $13^3=3^2\cdot 237+3\cdot21 +1$, we have $3\in\widetilde\cN_{3,237}$. 
On the other hand, direct computation 
shows that there is no  elliptic curve over any finite field $\F_{p^3}$
whose group of points $E(\F_{p^3})$ isomorphic to $\Z_3\times\Z_{3\cdot237}$. In fact,
the equation $p^3=3^2\cdot 237+3\,\ell +1$ with $|\ell|<2\sqrt{237}=30.79\cdots$ admits 
only one solution $(p,\ell)=(13,21)$, and $13^3+1-9\cdot237=5\cdot13$ is
not a value for the parameter $a$ that is permitted by Lemma~\ref{lem:Water}.

\subsection{Results with $k=1$}

Here, we focus on the problem of bounding $\#\cN_{m,1}(T)$.
We begin by quoting three results on Diophantine equations
due to Lebesgue~\cite{Leb}, to Nagell~\cite{Na}, and to Ljunggren~\cite{L}, respectively.

\begin{lemma} 
\label{lem:Dioph0} For any $m\in\N$, the Diophantine equation
$y^m=x^2+1$ has only the trivial solutions $(0,\pm 1)$.
\end{lemma}

\begin{lemma} 
\label{lem:Dioph1} For any $m\in\N$ that is not a power of three,
the Diophantine equations $y^m=x^2+x+1$ and $y^m=x^2-x+1$
have only trivial solutions from the set $\{(0,\pm 1),(\pm 1,\pm 1)\}$.
\end{lemma}

\begin{lemma} 
\label{lem:Dioph2} 
The only solutions of the Diophantine equation $y^3=x^2+x+1$
are the following: $\{(0,\pm 1),(-1,\pm 1),(18,7),(-19,7)\}$.
\end{lemma}

The main result here is the following:

\begin{theorem}
\label{eq:Nm1}
If $m$ is even, then
$$
\#\cN_{m,1}(T)=(m+o(1))\,\frac{T^{2/m}}{\log T}\qquad(T\to\infty).
$$
If $m\ge 5$ and $m$ is odd, then $\cN_{m,1}=\varnothing$. Also,
$\cN_{3,1}=\{18,19\}$, and
$$
\#\cN_{1,1}(T)\ll\frac T{\log T}\,.
$$
\end{theorem}

\begin{proof}
First, suppose that $m=2r\ge 2$ and $n\in\cN_{m,1}$.
Then there exists a prime $p$ such that 
$$
p^{2r}=n^2+\ell n+1\qquad\text{for some~}\ell\in\{0,\pm1,\pm2\}.
$$
However, the cases $\ell\in\{0,\pm1\}$ can be excluded in view of Lemmas~\ref{lem:Dioph0}
and~\ref{lem:Dioph1}.
Since the numbers $n$ for which this relation holds with $\ell\in\{\pm 2\}$ are those
of the form $n=p^r\pm1$, by the prime number theorem it follows that
$$
\#\{n\le T~:~n=p^r\pm1\}
=(2+o(1))\,\frac{T^{1/r}}{\log T^{1/r}}
=(m+o(1))\,\frac{T^{2/m}}{\log T}\,,
$$
and the proof is complete when $m$ is even.

Next suppose that $m=2r+1\ge 5$. 
Combining Lemmas~\ref{lem:Dioph0},~\ref{lem:Dioph1} and~\ref{lem:Dioph2},
one sees that there is no integer $n$ for which any one of the numbers $n^2+1$,
$n^2+n+1$, or $n^2-n+1$ is the $m$-th power of a prime.  Since 
the relation $(n\pm 1)^2=p^{2r+1}$ is also impossible, it follows
$\cN_{m,1}=\varnothing$ as stated.
 
When $m=3$ we are lead to consider the three Diophantine equations
$$
y^3=x^2+1, \qquad y^3=x^2+x+1\mand y^3=x^2-x+1.
$$
The first equation has no nontrivial solution by Lemma~\ref{lem:Dioph0},
the second only the nontrivial solution $(18,7)$ by Lemma~\ref{lem:Dioph2},
and the third only the nontrivial solution $(19,7)$ by Lemma~\ref{lem:Dioph2}.
Since $\gcd(7,20)=\gcd(7,-17)=1$, using Lemmas~\ref{lem:Water}
and~\ref{lem:Ruck} we conclude that $\cN_{3,1}=\{18,19\}$.

As an application of Theorem~\ref{thm:Nk1}, we deduce that
$$
\cN_{1,1}(T)=\{n\le T~:~n^2+1, n^2+n+1,\text{~or~}n^2-n+1\text{~is prime}\}.
$$
Using Brun sieve (see~\cite[Chapter~I.4, Theorem~3]{Ten}) or the Selberg sieve
(see~\eqref{eq:SelbergSieve} in \S\ref{sec:SpSq}) we see that
$\#\cN_{1,1}(T)\ll T/\log T$ as required.
\end{proof}

\begin{remark} 
Recalling the asymptotic version of 
Schinzel's \emph{Hypothesis~H} (see~\cite{SS})
given by Bateman and Horn~\cite{BH}, it is reasonable to conjecture that
$$
\#\cN_{1,1}(T)=(C+o(1))\,\frac{T}{\log T}\qquad(T\to\infty),
$$
where
$$
C=\frac{1}{2}\prod_{p\ge 3}\(1-
\frac{\bigl(\tfrac{-1}{p}\bigr)}{p-1}\)
+\prod_{p\ge 3}\(1-\frac{\bigl(\tfrac{-3}{p}\bigr)}{p-1}\)
$$
and $\big(\tfrac{\cdot}{p}\big)$ is the Legendre symbol modulo $p$.
We note that two distinct polynomials are simultaneously prime
for $O(T/(\log T)^2)$ arguments $n\le T$,
so we simply estimate the number of prime values for each of the above polynomials
independently. 
\end{remark}

\subsection{Finiteness of $\cN_{m,k}$ when $m\ge 3$}

In this section, we set
$$
\cK_k=\bigcup_{m\ge 3}\cN_{m,k}\mand
\cM_m=\bigcup_{k\ge 1}\cN_{m,k}.
$$
We show that there are only finitely many prime powers $p^m$ with $m\ge 3$
for which there is an elliptic curve $E$ defined over $\F_{p^m}$
with $E(\F_{p^m})\cong\Z_n\times\Z_{kn}$ for some $n\in\N$.
In other words, we have:

\begin{theorem}
\label{thm:lelenita} For every $k\ge 2$ the set $\cK_k$ is finite.
\end{theorem}

\begin{proof}
We apply a result of Schinzel and Tijdeman~\cite{ST} which asserts
that if a polynomial $f$ with rational coefficients
has at least two distinct zeros, then the equation 
$y^m=f(x)$, where $x$ and $y$ are integers with $y\ne 0$,
implies that $m\le c(f)$, where $c(f)$ is a computable constant that depends
only on $f$. 

For any $n\in\cK_k$, there exists a prime $p$ and integers $m,\ell$
with $m\ge 3$ and $|\ell|\le 2\sqrt{k}$ such that
$p^m=kn^2+\ell n+1$.  

For values of $\ell$ with $|\ell|<2\sqrt{k}$, the
polynomial $kX^2+\ell X+1$ has distinct roots. 
Thus we apply a result of Schinzel and Tijdeman~\cite{ST} 
which asserts
that if a polynomial $f$ with rational coefficients
has at least two distinct zeros, then the equation 
$y^m=f(x)$, where $x$ and $y$ are integers with $y\ne 0$,
implies that $m\le c(f)$, where $c(f)$ is a computable constant that depends only on $f$, see also~\cite[Theorem~10.2]{ShTi}. Hence, 
there are only finitely
many possibilities for the number $m$.  
For any fixed pair $(m,\ell)$, using a classical result in the 
theory of Diophantine equations (see~\cite[Theorem~6.1]{ShTi}), 
we conclude that there are only finitely many possibilities for the pair $(n,p)$.

If $\ell=\pm2\sqrt{k}$, then $k=h^2$
is a perfect square, and we have $p^m=(hn\pm1)^2$.
Thus, $m$ is even, and $h^2n^2=p^m+1-a$, where $a=\pm 2p^{m/2}$.
Applying Lemma~\ref{lem:Ruck}$\,(i)$ it follows that $kn=h^2n=n$;
this contradicts our hypothesis that $k\ge 2$ and shows that the case
$\ell=\pm2\sqrt{k}$ does not occur.
\end{proof}

\begin{remark}
All of the underlying ingredients in the proof of Theorem~\ref{thm:lelenita}
are effective, so one can easily obtain explicit bounds on 
$\#\cK_k$ and $\max\{n \in \cK_k\}$.
Using the explicit estimates of Bugeaud~\cite[Theorem~2]{Bu}, it can be shown
that $\cN_{m,k}=\varnothing$ for any $m>2^{137}k^{3/2}(\log_24k)^6$.
Further a result of Bugeaud~\cite[Theorem~2]{Bu2} on 
solutions of superelliptic equations imply the bound $\max\{n \in \cN_{m,k}\}\le\exp\(c(m)k^{14m}(\log k)^{8m}\)$,
where $c(m)$ is an effectively computable constant that depends only on $m$.
\end{remark}

A computer search suggests that 
the following table  lists completely 
the elements in $\cK_k$ for $2\le k\le 5$:

\begin{small}
\begin{center}
\begin{tabular}{|c|c|l|}
\hline
$k$ & $\cK_k$ & \\
\hline
$2$ & $\{3,11,45,119,120\}$& $2^4=2\cdot 3^2-3+1$,\\
    &                      & $3^5=2\cdot 11^2+1$,\\
    &                      & $2^{12}=2\cdot 45^2+45+1$,\\
    &                      & $13^4=2\cdot 119^2+2\cdot119+1$,\\
    &                      &$13^4=2\cdot 120^2-2\cdot120+1$. \\
$3$ & $\{5,72,555\}$      & $3^4=3\cdot 5^2+5+1$,\\
    &                     & $5^6=3\cdot 72^2+72+1$,\\
    &                     & $31^4=3\cdot 555^2-555+1$.\\
4 &  $\{1,9,23\}$ & $2^3=4\cdot1^2+3\cdot1+1$,\\ & & $7^3=4\cdot9^2+2\cdot9+1$,\\
  &               & $2^{11}=4\cdot23^2-3\cdot23+1$.  \\
5 & $\{1,2,4,56, 126\}$& $2^3=5\cdot1^2+2\cdot1+1$,\\ & &$3^3=5\cdot2^2+3\cdot2+1$,\\
  &                      &$3^4=5\cdot4^2+1$,\\ & &$5^6=5\cdot56^2-56+1$,\\
  &                      &$43^3=5\cdot126^2+126+1$.\\
\hline
\end{tabular}
\end{center}
\end{small}

\begin{theorem}
\label{thm:elena} For every natural number $m$ we have $\mathcal M_m=\N$. In other words,
for any $n,m\in \N$ there is a prime $p$ and an elliptic curve $E$ defined over $\F_{p^m}$ such that
$E(\F_{p^m})\cong \Z_n\times\Z_{kn}$ for some $k\in\N$.
\end{theorem}

\begin{proof} Let $m\in\N$ be fixed. 
If $m\ge 2$, then we have the identity
$$
X^m=(X^{m-2}+2X^{m-3}+\cdots+(m-2)X+m-1)(X-1)^2+m(X-1)+1.
$$
For any $n\in\N$, let $p$ be a prime in the arithmetic progression $1\bmod n$ 
that does not divide $m$, and put $d=(p-1)/n$. Applying the above identity
with $X=p$, we have $p^m=kn^2+\ell n+1$, where
$$
k=(p^{m-2}+2p^{m-3}+\cdots+(m-2)p+m-1)d^2\mand \ell=md.
$$ 
The condition $|\ell|\le 2\sqrt{k}$ is easily verified since
$$
4k\ge 2m(m-1)d^2\ge m^2d^2=\ell^2\qquad(m\ge 2).
$$
Furthermore, $a=p^m+1-kn^2=\ell n+2=m(p-1)$ is not divisible by $p$.
Hence, Lemma~\ref{lem:Ruck} shows that $n\in\cM_m$.

If $m=1$, then for any $n\in\N$, let $p$ be an odd prime in the arithmetic
progression $1\bmod n^2$. Then $p=dn^2+1$ for some natural number $d$,
and since $a=p+1-dn^2=2$ is not divisible by $p$, Lemma~\ref{lem:Ruck} shows
that $n\in\cM_1$. 
\end{proof}

\section{Missed group structures}

We have already given in~\eqref{eq:misses} 
several examples of pairs $(n,k)$
for which the group $\Z_n\times\Z_{kn}$ cannot be realized
as the group of points on an elliptic curve defined over a finite field.

Here we present more extensive numerical results.

In Figure~\ref{fig:fN} we plot the counting function 
$$
f(D) = D^2 - \#\Sq(D,D) 
$$ 
of ``missed'' pairs $(n,k)$ with $\max\{n,k\}\le D$ for 
values of $D$ up to $37550$. We immediately 
derive from Corollary~\ref{cor:Sq-bound} that 
$$
\lim_{D \to \infty} f(D)/D=\infty,
$$ but this statement seems weak in view
of our computations. 

\begin{figure}[htb]
\begin{center}
\includegraphics*[width=\textwidth]{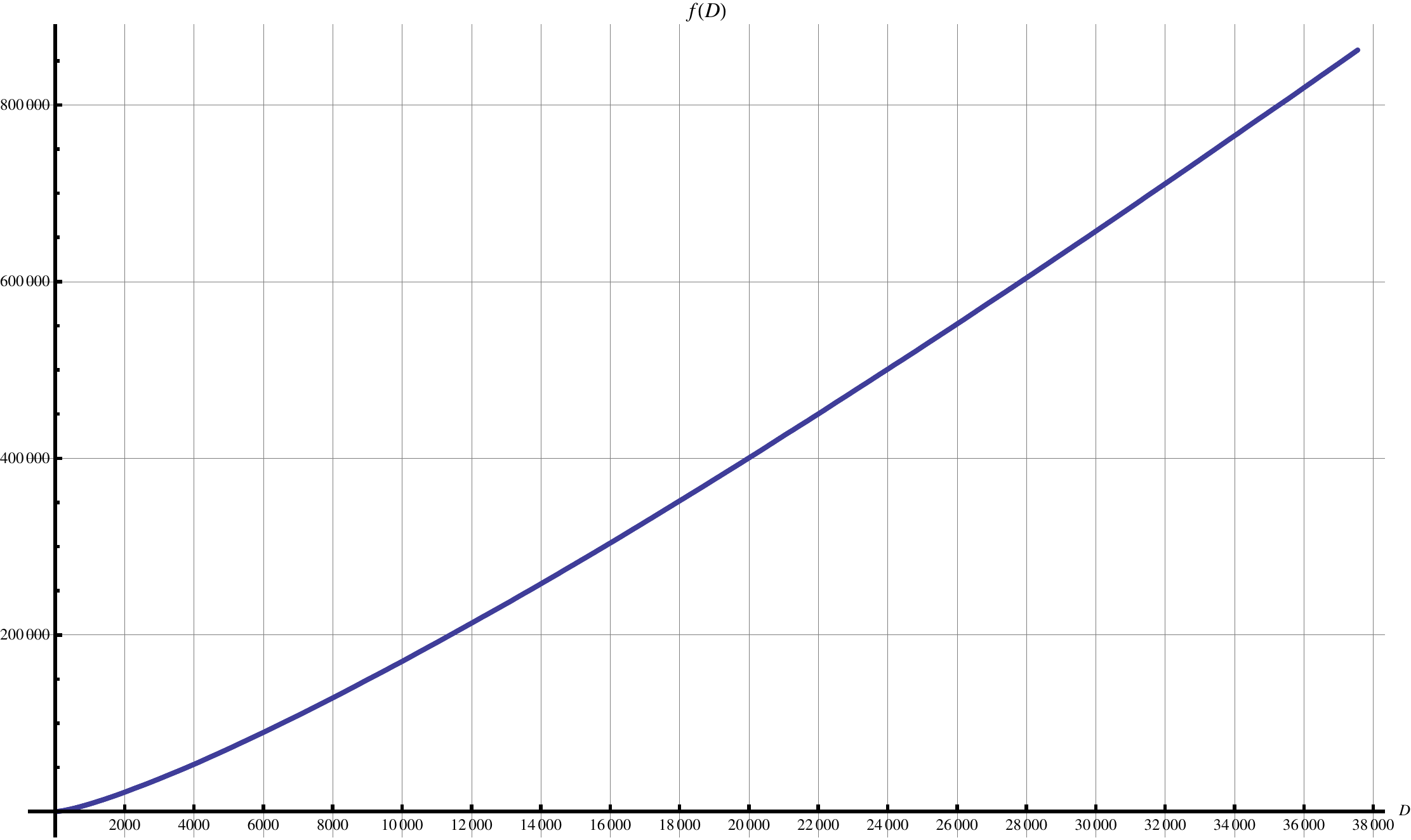}
\end{center}
\vspace*{-10mm} \caption{Plot of $f(D)$ for $D\le 37550$} \label{fig:fN}
\end{figure}

In Figure~\ref{fig:FNK} we  plot the counting function 
$$
F(N,K) = NK -  \#\Sq(N,K) 
$$ 
of ``missed'' pairs $(n,k)$ with $n\le N$ and $k\le K$ for values of
$N$ and $K$ up to $1000$. For each fixed $N=N_0$ the function
$G_{N_0}(K)=F(N_0,K)$ appears to be linear and increasing for modest values of $K$.
Clearly,  Corollary~\ref{cor:Sq-bound}  implies that when $K=K_0$ 
is fixed then $H_{K_0}(N) = F(N,K_0) \sim K_0 N$ grows 
asymptotically linearly with the coefficient $K_0$.

\begin{figure}[htb]
\begin{center}
\includegraphics*[width=\textwidth]{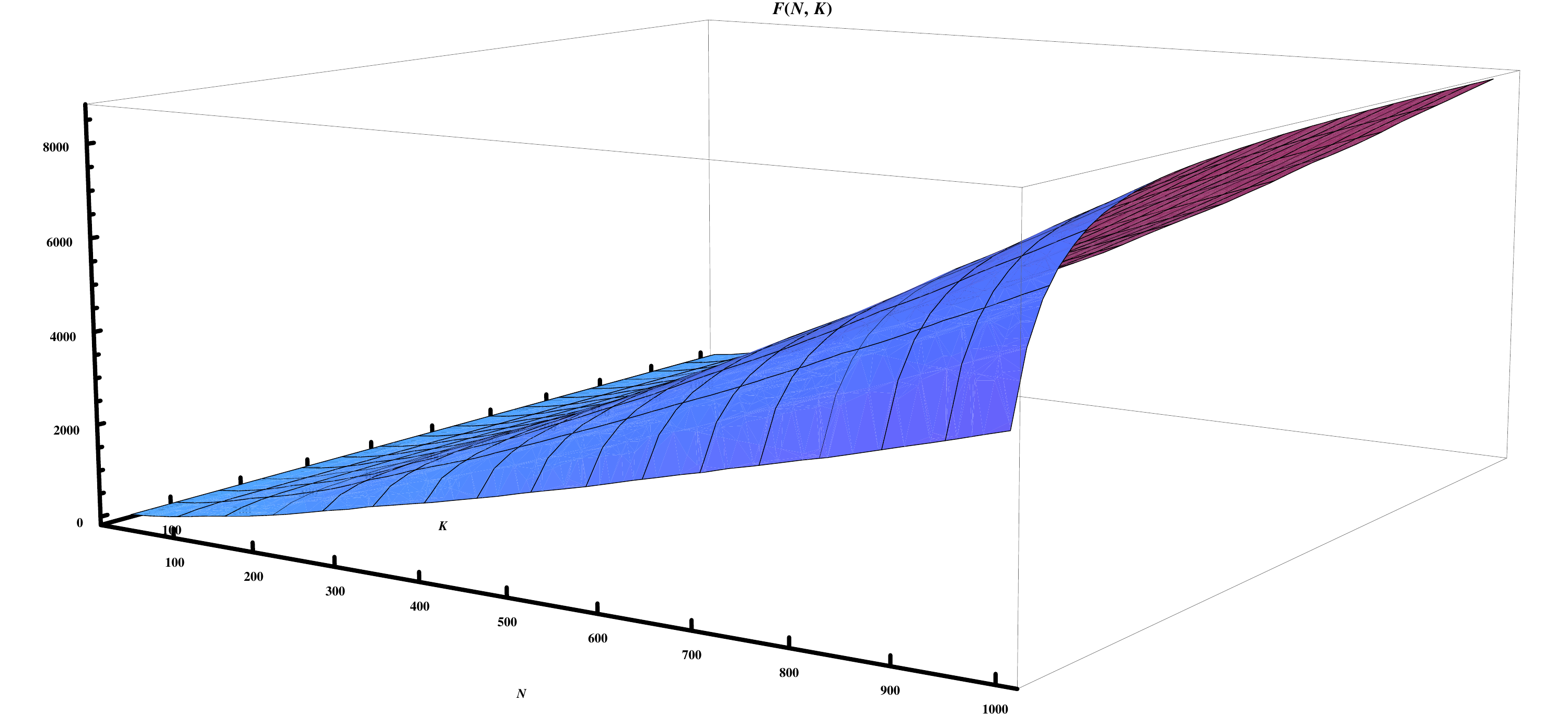}
\end{center}
\vspace*{-10mm} \caption{3D plot of $F(N,K)$ for $N,K \le 1000$} \label{fig:FNK}
\end{figure}

We now give some heuristic arguments to predict the
behavior of $F(N,K)$. We note that a pair $(n,k)$ contributes to $F(N,K)$ 
if $kn^2+\ell n+1$ is not a prime power for every $\ell$ such that  
$|\ell|\le 2k^{1/2}$ (and in some other exceptional cases).
Following the standard heuristic, $kn^2+\ell n+1$ is a prime 
power with   ``probability'' about
$$
\rho(n,k, \ell) = \begin{cases}
\displaystyle\frac{n}{\varphi(n) \log(kn^2+\ell n+1)}\, & \text{if }kn^2+\ell n+1>1\\
0 & \text{otherwise}.\end{cases}
$$
(where the ratio $n/\varphi(n)$ accounts for the fact that 
we seek prime powers in the arithmetic progression $1 \bmod n$).  
So $(n,k) \in [1,N]\times [1,K]$ contributes to $F(N,K)$ 
 with   ``probability'' about
$$
\vartheta(n,k) =\prod_{|\ell| \le 2k^{1/2}}\(1 - \rho(n,k, \ell)\) 
%= \prod_{|\ell| \le 2k^{1/2}} \(1 -\frac{n}{\varphi(n) \log(kn^2+\ell n+1)}\).
$$
Thus, we expect that $F(N,K)$ is close to
$$
B(N,K)=\sum_{n\le N}\sum_{k\le K}\vartheta(n,k).
$$
We have not studied the function $B(N,K)$ analytically, but we
note that for any fixed $\eps>0$ we have
$$
\vartheta(n,k)\approx
\begin{cases}
1&\quad\hbox{if $k\le(\log n)^{2-\eps}$,}\\
0&\quad\hbox{if $k\ge(\log n)^{2+\eps}$.}
\end{cases}
$$
Thus, it seems reasonable to expect that
$$
F(N,K)\approx B(N,K)\approx
\begin{cases}
NK&\quad\hbox{if $K\le(\log N)^{2-\eps}$,}\\
%N(\log N)^{2+o(1)}&\quad\hbox{if $K\ge(\log N)^{2+\eps}$.}
o(NK)&\quad\hbox{if $K\ge(\log N)^{2+\eps}$.}
\end{cases}
$$
One can see on Figure~\ref{fig:3DGuess} that 
the ratio
$$
\beta(N,K) = \frac{F(N,K)}{B(N,K)}
$$
seems to stabilise when $N$ and $K$ are large enough. 
\begin{figure}[htb]
\begin{center}
\includegraphics*[width=\textwidth]{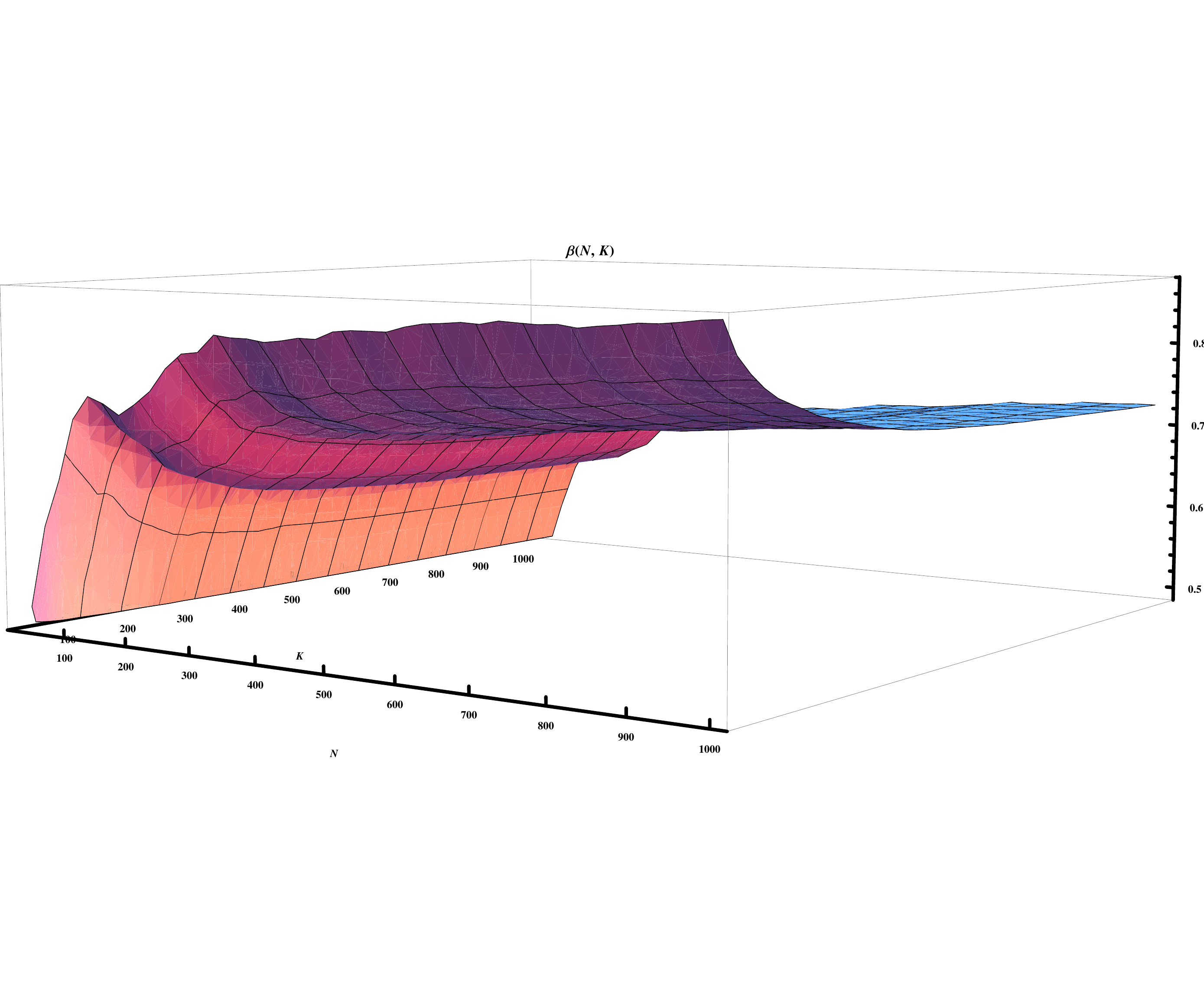}
\end{center}
\vspace*{-10mm} \caption{3D plot of $\beta(N,K)$ for $N,K \le 1000$} \label{fig:3DGuess}
\end{figure}

%Finally, in our original attempt to understand the behavior of $f(D)$, 
%we explored whether it has the form
%\begin{equation}
%\label{eq:f-guess}
%f(D) =D^{\gamma+o(1)}
%\end{equation}
%for some constant $\gamma>0$ by plotting the function
%$$
%\gamma(D)=\frac{\log f(D)}{\log D}
%$$
%for values of $D$ up to $37550$; see Figure~\ref{fig:Guess}.
%\begin{figure}[htb]
%\begin{center}
%\includegraphics*[width=\textwidth]{EC-Groups_Guess}
%\end{center}
%\vspace*{-10mm} \caption{Plot of $\gamma(D)$ for $D \le 37550$} \label{fig:Guess}
%\end{figure}
%Unfortunately, our results are inconclusive. Although
%$\gamma(D)$ shows some signs of slowing down, within the
%range of our computations the function continues to decrease
%monotonically, which suggests that $f(D)$ may not of
%%the form~\eqref{eq:f-guess} for any fixed $\gamma$.  
%A more detailed theoretical and numerical investigation
%of this function would be of ultimate interest. 

\section*{Acknowledgements}
The authors are grateful to Karl Dilcher for pointing 
out the relevance of the result of Ljunggren~\cite{L}
to this work, to Andrzej Schinzel for a discussion
concerning Cram\'er's Conjecture for arithmetic progressions
and to Corrado Falcolini for his help with Mathematica Plotting.

The second author was partially supported by GNSAGA from INDAM.
The third author was supported in part by ARC Grant DP0881473, Australia 
and by NRF Grant~CRP2-2007-03, Singapore.

\end{document}